\documentclass[11pt]{article}

\usepackage[normalem]{ulem} 
\usepackage{setspace}
\usepackage{color}
\usepackage{subfig}
\usepackage{enumerate}
\usepackage{graphicx}
\usepackage{amssymb}
\usepackage{amsthm}
\usepackage{amsmath}
\usepackage{bm}
\usepackage{epstopdf}
\usepackage{authblk}
\usepackage{epsfig}
\usepackage{cite}
\usepackage{palatino}
\usepackage{microtype}
\usepackage{ulem}
\usepackage{multirow}
\usepackage{bbm}
\usepackage{cleveref}
\usepackage{datetime}
\usepackage{bbm}

\DeclareMathOperator*{\argmin}{argmin}

\newtheorem{thm}{Theorem}

\newtheorem{algorithm}{Algorithm}

\title{Sampling error correction in ensemble Kalman inversion}
\author[1]{Yoonsang Lee\thanks{yoonsang.lee@dartmouth.edu}}
\affil[1]{Department of Mathematics, Dartmouth College}
\date{}

\begin{document}

\maketitle
\begin{abstract}
Ensemble Kalman inversion is a parallelizable derivative-free method to solve inverse problems. The method uses an ensemble that follows the Kalman update formula iteratively to solve an optimization problem. The ensemble size is crucial to capture the correct statistical information in estimating the unknown variable of interest. Still, the ensemble is limited to a size smaller than the unknown variable's dimension for computational efficiency. This study proposes a strategy to correct the sampling error due to a small ensemble size, which improves the performance of the ensemble Kalman inversion.
This study validates the efficiency and robustness of the proposed strategy through a suite of numerical tests, including compressive sensing, image deblurring, parameter estimation of a nonlinear dynamical system, and a PDE-constrained inverse problem.
\end{abstract}

\section{Introduction}\label{sec:introduction}
Ensemble-based Kalman filters \cite{EnKF,EnSQKF,EAKF,ETKF} have shown successful results in estimating state variables or parameters, particularly for high-dimensional systems, such as oceanic and atmospheric applications.
As a data assimilation method that incorporates information from a prediction model and measurement data, the ensemble Kalman filter focuses on an accurate statistical description of a variable of interest that changes over time.
Ensemble Kalman inversion (EKI), pioneered in the oil industry and mathematically founded in \cite{EKI}, uses the ensemble Kalman update iteratively to solve an inverse problem as a derivative-free and parallelizable method. 
EKI uses the statistical information of the ensemble to determine the direction of an update without variational calculations of the gradients of the forward or the adjoint models of an inverse problem.
Also, each ensemble member is independent of each other in evaluating the forward model, which enables a highly efficient parallel computation. 
From these characteristics of EKI, EKI has been applied to a wide range of inverse problems where a forward model is computationally expensive to solve. 
Examples include the estimation of the permeability field of subsurface flow \cite{iterativeregularization}, learning of dynamical systems with noise \cite{sparseEKI}, and machine learning tasks \cite{MLEKI}. 

In EKI, a small ensemble size introduces statistical deficiencies in extracting the statistical information of the ensemble, which affects the accuracy of EKI. However, for an inverse problem with a computationally expensive forward model, such as PDE-constrained optimization problems, the ensemble size is limited to a small number (significantly smaller than the unknown variable's dimension) due to the computational cost of running the forward model. The ensemble updated by the Kalman formula remains in the linear span of the initial ensemble \cite{EKI}, which plays a role in the regularization of EKI. Therefore, the initial ensemble must be chosen judiciously so that the unknown variable lies in the linear span of the initial ensemble, which can be challenging without prior information of the unknown variable. For sparse reconstruction problems, the multiple batch idea \cite{multipleBatch} has been incorporated to enable EKI to use a small ensemble size \cite{lpEKI,sparseEKI} without relying on the ensemble initialization. The multiple batch idea runs several batches where unnecessary components (that is, zero components) are removed after each batch. Through this process, the problem can be formulated as a smaller size problem than the original problem. However, the multiple batch approach for EKI does not work for non-sparse recovery problems.

\begin{figure}[!t]
\centering
\includegraphics[width=.5\textwidth]{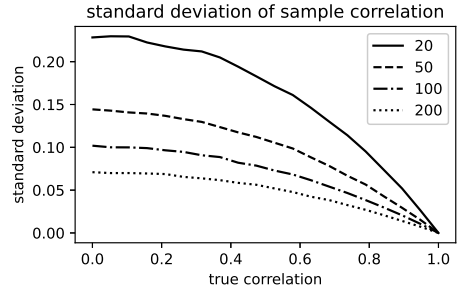}
\caption{Standard deviations of sample correlations for various true correlation values and sample sizes (different lines represent different sample sizes).}\label{fig:sampleerror}
\end{figure}
Ensemble-based Kalman filters also have the sampling issue, which has attracted many researchers in the data assimilation community, particularly for geophysical systems. Due to the vast size of a geophysical system, the unknown variable's dimension is typically much larger than the ensemble size. Common strategies to handle the sampling error include covariance inflation \cite{inflation} and localization \cite{locHoutekamer, locHamill}. The covariance inflation increases the ensemble's spread to account for the missing information in the ensemble. The localization, on the other hand, is related to modifying the spurious correlation of the ensemble. 
Figure \ref{fig:sampleerror} shows the standard deviation of sample correlations using various sample sizes and true correlation values. This figure shows that the sampling error in the sample correlation is significant for a small true correlation and a small sample size. When the sample size is small, there is a high chance of a spurious correlation between unrelated components, and thus the ensemble filters lose accuracy by making unnecessary updates.
The idea of the covariance localization is to shrink the spurious correlation based on the distance between the measurement data and the unknown variable of interest. The localization, along with the inflation, has played a crucial role in the accuracy of the ensemble-based Kalman filters, including its application for a Bayesian inverse method using Markov-chain Monte Carlo \cite{locMCMC}.

The covariance inflation and localization have been incorporated in EKI \cite{analysisEKI}. Still, the localization is limited when there is a well-defined distance between the measurement data and the unknown variable of interest. For a nonlinear forward map where the relation between the measurement data and the unknown variable of interest is not explicit, such a localization method is not applicable. There is a localization method for nonlinear observation operators \cite{locEKF} based on the Schur-product of the covariance. However, this method is still limited to modifying correlations between the components of the unknown variable where a distance is well-defined. 

The main goal of this study is to propose a sampling error correction (SEC) strategy that enables a small ensemble size in EKI for inverse problems in which there is no well-defined distance for the covariance localization. The strategy is influenced by the sampling error correction method for ensemble Kalman filters \cite{AndersonSEC}. The SEC method in \cite{AndersonSEC} uses a correlation correction term built from offline Monte Carlo simulations for various ensemble sizes and correlations. The correlation correction is significant for a very small sample size. Still, the correction attenuates for a sample size of order $\mathcal{O}(100)$ or larger, which is still small for a high-dimensional unknown variable to estimate. 
Thus, a direct application of SEC in \cite{AndersonSEC} to EKI does not significantly impact performance when the ensemble size is comparable to or larger than $\mathcal{O}(100)$.
Our strategy uses a power function to impose a much stronger correction for various ensemble sizes.
The calculation of the correction term is explicit and does not require any pre-computation. As it is known for the localization in ensemble Kalman filters, we show in Section \ref{sec:secEKI} that the ensemble updated through the sampling error correction does not necessarily lie in the linear span of the initial ensemble. That is, the sampling error correction overcomes the limitation of the invariant subspace property of the standard EKI. Thus, a judicious ensemble initialization is not required for the robust performance of EKI using a small ensemble size.

The structure of this paper is as follows. In Section \ref{sec:EKI}, we briefly review the standard EKI. Section \ref{sec:secEKI} is devoted to the description of the sampling error correction method for EKI and an analysis of its property. A suite of numerical evidence is provided in Section \ref{sec:numerical}, including compressive sensing, image deblurring, parameter estimation of a nonlinear dynamical system, and a 2D PDE-constrained inverse problem. The paper concludes in Section \ref{sec:conclusions} with discussions about limitations and future directions of the current study.

\section{Ensemble Kalman inversion}\label{sec:EKI}
In this section, we review the discrete-time ensemble Kalman inversion (EKI). The contents of this section are intended to be minimal for the description of the sampling error correction method for EKI proposed in the next section. For more details about the method in this section, we refer \cite{EKI} and \cite{lpEKI}.

\subsection{Problem formulation}
Before we describe the discrete-time EKI, we formulate an inverse problem of interest. Using EKI, we are interested in estimating an unknown variable $u\in\mathbb{R}^N$ using measurement data $y\in\mathbb{R}^{M}$ where $u$ and $y$ are related as
\begin{equation}\label{eq:inverseproblem}
y = G(u)+\eta.
\end{equation}
$G:\mathbb{R}^N\to\mathbb{R}^{M}$ is a forward model, which can be nonlinear and computationally expensive to solve, such as solving a PDE model. $\eta\in\mathbb{R}^{M}$ is a measurement error, which is assumed to be Gaussian with mean zero and covariance $\Gamma$. 
With an appropriate regularization term $R(u)$, the unknown variable $u$ is estimated by solving the following optimization problem
\begin{equation}\label{eq:optimizationproblem}
\argmin_{u\in\mathbb{R}^N}\; R(u)+\|y-G(u)\|_{\Gamma}^2.
\end{equation}
Here $\|\cdot\|_{\Gamma}$ is the norm induced by the inner product using the inverse of the covariance matrix $\Gamma$. That is, $\|a\|_{\Gamma}^2=\langle a,\Gamma^{-1}a\rangle$ for the standard inner product $\langle,\rangle$ in $\mathbb{R}^{M}$. The $l_p$-regularized EKI ($l_p$EKI, \cite{lpEKI}) implements $R(u)=\lambda\|u\|_p^p,p>0$, including $p=2$ for Tikhonov regularization and $p\leq 1$ for sparse reconstruction. Here $\lambda$ is a regularization coefficient, which determines the strength of the regularization term. In the next section, we review the discrete-time EKI without the regularization term for its simplicity. We provide the complete $l_p$EKI algorithm in Appendix.

\subsection{Discrete-time ensemble Kalman inversion}\label{subsec:EKI}
In order to use the ensemble Kalman update for a nonlinear measurement problem, EKI uses an extended space $(u,g)\in\mathbb{R}^{N+M}$ and applies an artificial dynamics to the extended variable. For the variable at the ($n$-1)-th iteration step, $(u_{n-1},g_{n-1})$, the artificial dynamics applied to $(u_{n-1},g_{n-1})$ yields $(u_n,g_n)$ that is given by
\begin{equation}
(u_n,g_n)=(u_{n-1},G(u_{n-1})).
\end{equation}
The extended space framework enables us to use the ensemble Kalman update designed for linear measurement. The measurement in the extended framework is the projection of $(u,g)$ onto $g$. Instead of handling the nonlinear forward model as a measurement operator in the Kalman filter, we treat the forward model as a prediction model for the Kalman filter and apply the linear projection measurement for the extended state variable. Under this setup, EKI applies the artificial dynamics and the ensemble Kalman update iteratively to estimate $u$ as a solution of the inverse problem \eqref{eq:inverseproblem}.

In implementing the discrete-time EKI, it is not necessary to handle the higher dimensional state vector $(u,g)$. As we are interested in the update of $u$ while $g$ is subordinate to $u$, the actual iteration of EKI does not involve any computation related to a vector in $\mathbb{R}^{N+M}$. The complete discrete-time EKI algorithm is described below.

\vspace{0.05\textwidth}
\begin{algorithm}\label{algo:EKI}
\textbf{discrete-time EKI}
\end{algorithm}
Assumption: an initial ensemble of size $K$, $\{u_0^{(k)}\}_{k=1}^K$ from prior information, is given.\\
For $n=1,2,...,$
\begin{enumerate}
	\item Prediction step using the artificial dynamics:
	\begin{enumerate}
		\item Apply the forward model $G$ to each ensemble member 
\begin{equation}
g_n^{(k)}:=G(u_{n-1}^{(k)})
\end{equation}
		\item From the set of the predictions $\{g_n^{(k)}\}_{k=1}^K$, calculate the mean and covariances
\begin{equation}\label{eq:samplemean}
\overline{g}_n=\frac{1}{K}\sum_{k=1}^Kg_n^{(k)},
\end{equation}
\begin{equation}\label{eq:samplecovariance}
\begin{split}
C^{ug}_n&=\frac{1}{K}\sum_{k=1}^K(u_n^{(k)}-\overline{u}_n)\otimes(g_n^{(k)}-\overline{g}_n),\\
C^{gg}_n&=\frac{1}{K}\sum_{k=1}^K(g_n^{(k)}-\overline{g}_n)\otimes(g_n^{(k)}-\overline{g}_n),
\end{split}
\end{equation}

where $\overline{u}_n$ is the mean of $\{u_n^{(k)}\}$, i.e., $\displaystyle\frac{1}{K}\sum_{k=1}^Ku_n^{(k)}$.
	\end{enumerate}

\item Analysis step:
	\begin{enumerate}
		\item Update each ensemble member $u_n^{(k)}$ using the Kalman update
\begin{equation}\label{eq:ensembleupdate}
u_{n+1}^{(k)}=u_{n}^{(k)}+C^{ug}_n(C^{gg}_n+\Gamma)^{-1}(y_{n}^{(k)}-g_n^{(k)}),
\end{equation}
where $y_{n+1}^{(k)}=y+\zeta_{n+1}^{(k)}$ is a perturbed measurement using Gaussian noise $\zeta_{n+1}^{(k)}$ with mean zero and covariance $\Gamma$.

		\item Compute the ensemble mean as an estimate for the solution $u$
		\begin{equation}
		\overline{u}_{n+1}=\frac{1}{K}\sum_{k=1}^Ku_n^{(k)}
		\end{equation}
	\end{enumerate}
\end{enumerate}
We note that the term $C^{ug}_n(C^{gg}_n+\Gamma)^{-1}$ in \ref{eq:ensembleupdate} is from the Kalman gain matrix. The goal of EKI is to update $u$ to approximate the solution of the optimization problem \eqref{eq:optimizationproblem}. As we are interested in only $u$ while $g$ in the extended space is subject to $u$, a simplification of the standard Kalman update yields \eqref{eq:optimizationproblem}. 

The Kalman update strategy we consider here follows the idea of the ensemble Kalman filter \cite{EnKF}. Instead of this random update, it is possible to use the same measurement $y$ without perturbing it with Gaussian noise. 
Other classes of deterministic ensemble update methods are worth further consideration, such as such as square-root filters \cite{EnSQKF}, including ensemble adjustment \cite{EAKF} and ensemble transformation \cite{ETKF} Kalman filters. Also, in \cite{analysisEKI}, a continuous-time limit of EKI has been discussed, which has a computational benefit. It does not require inversion of a covariance matrix if the observation errors are uncorrelated between each component (that is, the observation error covariance is diagonal). Additionally, EKI can exploit an adaptive time-stepping to improve the performance similar to the adaptive learning rate in the machine learning tasks. We leave the investigation of these variants for further study, and we focus on the discrete-time perturbed observation update described in the algorithm.

\section{Sampling error correction}\label{sec:secEKI}
The standard localization for data assimilation requires a well-defined distance function between different components of variables (the measurement and the unknown variable). When there is a well-defined distance function, the localization attenuates the impact of observation on variables further than a cutoff distance that is a tunable parameter. One of the commonly used methods to calculate the correction term is the Gaspari-Cohn polynomial with compact support that approximates a normal distribution \cite{Gaspari-Cohn}.

The SEC strategy in \cite{AndersonSEC} handles the case in which there is no well-defined distance between different components of variables, such as satellite radiance. The idea is to modify the sample correlation by multiplying a correction term that depends on sample correlation and ensemble size. The correction factor ranges from 0 to 1, which decreases the magnitude of the sample correlation unless components are perfectly correlated. 
The SEC method uses the fact that the ensemble Kalman update is related to linear regression. By assuming that the linear regression coefficient follows a normal distribution with an unknown mean and variance, the correction term is derived as a minimizer of the expected RMS difference between the corrected regression coefficient and the maximum likelihood estimate of the coefficient. The corrected coefficient turns out to minimize the expected RMS error in the updated ensemble mean. As the correction term depends only on the sample correlation and size, for a fixed ensemble size, the method precomputes a lookup table for a possible range of sample correlation values through an offline Monte-Carlo computation.

\begin{figure}[!ht]
\centering
\includegraphics[width=.45\textwidth]{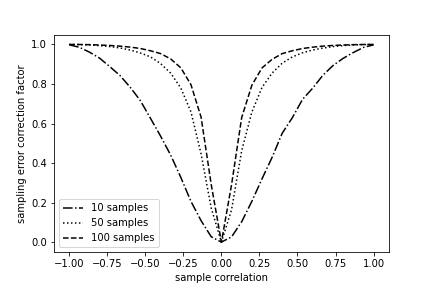}
\includegraphics[width=.45\textwidth]{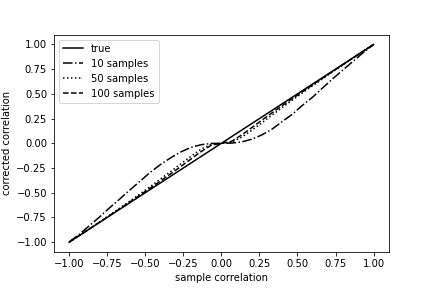}
\caption{Sampling error correction factors following \cite{AndersonSEC} (left) and their corresponding corrected correlations (right) for various sample sizes. }\label{fig:SEC}
\end{figure}
The left plot of Figure \ref{fig:SEC} shows the sampling error correction factor as a function of sample correlation and size following the method in \cite{AndersonSEC}, along with its corresponding corrected correlation value on the right plot. 
As the magnitude of the sample correlation decreases, the correction factor decreases monotonically from 1, and thus the SEC method yields a correlation weaker than the sample correlation. 
We note that as the ensemble size increases, the correction effect attenuates. In particular, the corrected correlation using 100 ensemble members is close to the sample correlation before correction. From this observation, we see that the sampling error correction is marginal for an ensemble size of an order of 100 or larger.

\subsection{Sampling error correction in ensemble Kalman inversion}\label{subsec:SECEKI}
It is straightforward to apply the SEC method in \cite{AndersonSEC} to EKI. For the sample covariances $C^{ug}_n$ and $C^{gg}_n$ in \eqref{eq:samplecovariance}, by suppressing $n$ for simplicity, we consider the following decompositions in the product form
\begin{equation}\label{eq:covdecomposition}
\begin{split}
C^{ug}=V^uR^{ug}V^g,\\
C^{gg}=V^gR^{gg}V^g
\end{split}
\end{equation}
where $V^u$ and $V^g$ are diagonal matrices with diagonal elements corresponding to the variance of each component of $u$ and $g$, respectively. $R^{ug}$ and $R^{gg}$ are the correlation matrices between $u$ and $g$
\begin{equation}
\begin{split}
\left(R^{ug}\right)_{ij}&=r^{ug}_{ij}=\mbox{sample correlation between } u_i \mbox{ and }g_j,\\
\left(R^{gg}\right)_{ij}&=r^{gg}_{ij}=\mbox{sample correlation between } g_i \mbox{ and }g_j.
\end{split}
\end{equation}

If $s=s(r, K)$ is the sampling error correction factor for sample correlation $r$ and ensemble size $K$, each element of $R^{ug}$ and $R^{gg}$ is corrected as 
\begin{equation}
r_{ij}\leftarrow s(r_{ij},K)r_{ij}.
\end{equation}
This modification yields sampling error-corrected correlation matrices $R^{ug}_{sec}$ and $R^{gg}_{sec}$ with $\left(R^{ug}_{sec}\right)_{ij}=s(r^{ug}_{ij},K)r^{ug}_{ij}$ and $\left(R^{gg}_{sec}\right)_{ij}=s(r^{gg}_{ij},K)r^{gg}_{ij}$. 
Using the corrected covariance matrices, the sampling error correction for EKI works like the standard EKI replacing the covariance matrices in \eqref{eq:samplecovariance} with the sampling error-corrected covariances $C^{ug}_{sec}$ and $C^{gg}_{sec}$ that are defined as
\begin{equation}
\begin{split}
C^{ug}_{sec}&=V^uR^{ug}_{sec}V^g,\\
C^{gg}_{sec}&=V^gR^{gg}_{sec}V^g.
\end{split}
\end{equation}

Our focus in this study is the design of the correction factor function so that the correction can be significant even for a moderate ensemble size. 
Motivated by the shape of the corrected correlation (the left plot of Figure \ref{fig:SEC}), we propose a power function as the correction factor function $s$ that depends only on the sample correlation $r$
\begin{equation}\label{eq:correctionftn}
s(r)=|r|^a, \quad a\geq 0,
\end{equation}
which weakens correlations for $a>0$ if the correlation is not perfect.
Instead of having dependence on the ensemble size, we leave the power $a$ as a tunable parameter that can account for both the ensemble size and the unknown variable's dimension.
Using the correction factor function \eqref{eq:correctionftn}, the sampling error-corrected correlation matrices $R^{ug}_{sec}$ and $R^{gg}_{sec}$ are given by multiplying the correction factor function to each element
\begin{equation}\label{eq:seccormatrix}
\begin{split}
\left(R^{ug}_{sec}\right)_{ij}=|r^{ug}_{ij}|^{a}r^{ug}_{ij},\\
\left(R^{gg}_{sec}\right)_{ij}=|r^{gg}_{ij}|^{a}r^{gg}_{ij}.
\end{split}
\end{equation}
We now have the complete sampling error correction algorithm for EKI.

\vspace{0.05\textwidth}
\begin{algorithm}\label{algo:SECEKI}
\textbf{EKI with sampling error correction}
\end{algorithm}
Assumption: an initial ensemble of size $K$, $\{u_0^{(k)}\}_{k=1}^K$ from prior information, is given. Also, the power $a$ for the sampling error correction factor function \eqref{eq:correctionftn} is given.\\
For $n=1,2,...,$
\begin{enumerate}
	\item Prediction step using the artificial dynamics:
	\begin{enumerate}
		\item Apply the forward model $G$ to each ensemble member 
\begin{equation}
g_n^{(k)}:=G(u_{n-1}^{(k)})
\end{equation}
		\item From the set of the predictions $\{g_n^{(k)}\}_{k=1}^K$, calculate the mean and covariances
\begin{equation}\label{eq:sec:samplemean}
\overline{g}_n=\frac{1}{K}\sum_{k=1}^Kg_n^{(k)},
\end{equation}
\begin{equation}\label{eq:sec:samplecovariance}
\begin{split}
C^{ug}_n&=\frac{1}{K}\sum_{k=1}^K(u_n^{(k)}-\overline{u}_n)\otimes(g_n^{(k)}-\overline{g}_n),\\
C^{gg}_n&=\frac{1}{K}\sum_{k=1}^K(g_n^{(k)}-\overline{g}_n)\otimes(g_n^{(k)}-\overline{g}_n),
\end{split}
\end{equation}
where $\overline{u}_n$ is the mean of $\{u_n^{(k)}\}$, i.e., $\displaystyle\frac{1}{K}\sum_{k=1}^Ku_n^{(k)}$.

	\item Using the product form decomposition of the covariance matrices \eqref{eq:covdecomposition}, calculate the sampling error-corrected covariance matrices
	\begin{equation}
	\begin{split}
	C^{ug}_{n,sec}&=V^uR^{ug}_{sec}V^g,\\
C^{gg}_{n,sec}&=V^gR^{gg}_{sec}V^g.
	\end{split}
	\end{equation}
	where $R^{ug}_{sec}$ and $R^{gg}_{sec}$ are given by \eqref{eq:seccormatrix}	\end{enumerate}

\item Analysis step:
	\begin{enumerate}
		\item Update each ensemble member $u_n^{(k)}$ using the Kalman update with the corrected covariance matrices
\begin{equation}\label{eq:sec:ensembleupdate}
u_{n+1}^{(k)}=u_{n}^{(k)}+C^{ug}_{n,sec}(C^{gg}_{n,sec}+\Gamma)^{-1}(y_{n}^{(k)}-g_n^{(k)}),
\end{equation}
where $y_{n+1}^{(k)}=y+\zeta_{n+1}^{(k)}$ is a perturbed measurement using Gaussian noise $\zeta_{n+1}^{(k)}$ with mean zero and covariance $\Gamma$.

		\item Compute the ensemble mean as an estimate for the solution $u$
		\begin{equation}
		\overline{u}_{n+1}=\frac{1}{K}\sum_{k=1}^Ku_n^{(k)}.
		\end{equation}
	\end{enumerate}
\end{enumerate}

\subsection{Violation of the invariant subspace property}
It is known for the ensemble Kalman update formula that the ensemble has an invariant subspace property. The ensemble updated through the ensemble Kalman update remains in the linear span of the initial ensemble \cite{EnKF}. This property can serve as a regularization in EKI \cite{EKI,TEKI}. However, when the ensemble size is smaller than the unknown variable's dimension, the invariance property requires an informed choice for the initial ensemble members so that the solution belongs to the linear span of the initial ensemble. As in the standard localization, we show that the ensemble updated by the sampling error correction for EKI does not necessarily satisfy the invariant subspace property.

\begin{thm}
Assume that the ensemble size is smaller than the unknown variable's dimension, i.e., $K<N$. 
The ensemble updated by the sampling error-corrected EKI belongs to the span of the previous ensemble if the sampling correlation $r^{ug}_{lm}$, the sample correlation between $\{u_l^{(k)}\}_{k=1,2,...,K}$ and $\{g_m^{(k)}\}_{k=1,2,...,K}$, is independent of $l$.
\end{thm}

\begin{proof}
For simplicity, we suppress the iteration index $n$. Following the notations used in Algorithm \ref{algo:EKI}, $u^{(k)}$ is the $k$-th ensemble member and $g^{(k)}=G(u^{(k)})$ is the measurement value using the $k$-th ensemble member. Let $\tilde{u}^{(k)}$ and $\tilde{g}^{(k)}$ be the deviations of $u^{(k)}$ and $g^{(k)}$ from their means, respectively. That is, $\tilde{u}^{(k)}=u^{(k)}-\frac{1}{K}\sum_{k=1}^K u^{(k)}$ and $\tilde{g}^{(k)} = g^{(k)}-\frac{1}{K}\sum_k g^{(k)}$.
The $lm$ element of the covariance matrix $C^{ug}$, $C^{ug}_{lm}$, has the following representation 
\begin{equation}
C^{ug}_{lm}=\sum_{k=1}^{K}\tilde{u}_l^{(k)}\tilde{g}_m^{(k)}=\sum_{k=1}^{K}u_l^{(k)}\tilde{g}_m^{(k)},
\end{equation}
where $l$ and $m$ represent the vector component index of the corresponding vectors.
Note that the second equality comes from the fact that $\sum_{k=1}^{K}\tilde{g}_m^{(k)}=0$.

Following the ensemble update using SEC \eqref{eq:sec:ensembleupdate}, the $l$-th component of the $j$-th ensemble member update is
\begin{equation}\label{eq:}
u_l^{(j)}\leftarrow u_l^{(j)} + \sum_{k=1}^K\sum_{m=1}^{M} u_l^{(k)}\gamma_{lm}\tilde{g}_m^{(k)}\hat{w}^{(j)}_m, \quad j=1,2,...,K, l=1,2,...,N
\end{equation}
where $\gamma_{lm}=(r_{lm}^{ug})^a$, $a\geq 0$, and $w^{(j)}_m$ is the $m$-th component of $\left(C^{gg}_{sec}+\Gamma\right)^{-1}\left(y^{(j)}-g^{(j)}\right)$.

Let $\beta_{l}^{(k)}=\sum_{m=1}^M \gamma_{lm}\tilde{g}_m^{(k)}w^{(j)}_m$. The updated ensemble member belongs to the linear span of the previous ensemble if and only if there exists $\alpha=(\alpha^{(1)},\alpha^{(2)},...,\alpha^{(K)})\in\mathbb{R}^K$ such that
\begin{equation}\label{eq:systemsofeqs}
\sum_{k=1}^{K}u_l^{(k)}\beta_l^{(k)}=\sum_{k=1}^{K}u_l^{(k)}\alpha^{(k)}, \quad l=1,2,...,N.\end{equation}
If $r^{ug}_{lm}$ is independent of $l$, $\beta_l^{(k)}$ is also independent of $l$, and thus $\alpha^{(k)}=\beta^{(k)}$ is the solution of the above system of equations.
\end{proof}
Note that \eqref{eq:systemsofeqs} is a system of $N$ equations of $K$ unknowns, $\{\alpha^{(k)}\}$. If the ensemble size is larger than the dimension of $u$, that is, $K>N$, and $\{u^{(k)}\}$ has rank $N$, the SEC updated ensemble belongs to the linear span of the previous ensemble regardless of the $l$-independence of $r^{ug}_{lm}$. As the left-hand side of \eqref{eq:systemsofeqs}, $\sum_{k=1}^{K}u_l^{(k)}\beta_l^{(k)}$, is a vector in $\mathbb{R}^N$, and $\{u^{(k)}\}$ has rank $N$, we can find $\alpha$ that satisfies \eqref{eq:systemsofeqs}.
When $K<N$, there is no guarantee that the SEC updated ensemble satisfies the invariant subspace property. The sample correlation $r^{ug}_{lm}$ typically depends on $l$, and thus the above theorem shows that the updated ensemble using SEC does not necessarily belong to the linear span of the previous ensemble.
$\beta^{(k)}_l$ depends on $\{u_l^{(k)}\}$, and thus there may be a hidden structure that guarantees a solution of \eqref{eq:systemsofeqs} for any sample correlation $r^{ug}_{lm}$. We provide the following example to illustrate how EKI with SEC works and that such a hidden structure does not exist in general.

\subsection*{Example: EKI update using the sampling error correction in $\mathbb{R}^4$}
This example considers a case when $N=4$, $M=1$, and $K=3$. The forward model is given by $G(u)=u_1$, the first component of $u$, and the measurement $y$ is equal to 2 with an observation error variance $\sigma_o^2=\frac{7}{9}$.
Let the three ensemble members are given as follows
\begin{equation}
u^{(1)}=\begin{pmatrix}1\\-1\\0\\0\end{pmatrix}, u^{(2)}=\begin{pmatrix}0\\1\\1\\0\end{pmatrix}, \mbox{and }u^{(3)}=\begin{pmatrix}0\\0\\0\\1\end{pmatrix}
\end{equation}
Therefore, we have the following measurement values of the ensemble members,
\begin{equation}
g^{(1)}=1,g^{(2)}=0,\mbox{ and }g^{(3)}=0.
\end{equation}
The covariance matrix $C^{ug}$ is $\begin{pmatrix}\frac{2}{9}\\\frac{-1}{3}\\\frac{-1}{9}\\\frac{-1}{9}\end{pmatrix}$ while the other covariance matrix, $C^{gg}=Var(\{g^{(k)}\})$, is $\frac{2}{9}$, a scalar value.

The correlation matrix corresponding to $C^{ug}$ is $\begin{pmatrix}1\\\frac{-\sqrt{3}}{2}\\\frac{-1}{2}\\\frac{-1}{2}\end{pmatrix}$, and thus the sampling error-corrected covariance matrices with $a=1$ are
\begin{equation}
C^{ug}_{sec}=\begin{pmatrix}\frac{2}{9}\\\frac{-\sqrt{3}}{6}\\\frac{-1}{18}\\\frac{-1}{18}\end{pmatrix}\quad\mbox{and}\quad
C^{gg}_{sec}=C^{gg}=\frac{2}{9}.
\end{equation}

For the update of the first ensemble member, $u^{(1)}$, we have
\begin{equation}
u^{(1)}\leftarrow u^{(1)}+C^{ug}_{sec}(C^{gg}_{sec}+\sigma_o^2)^{-1}(y-G(u^{(1)}))=u^{(1)}+C^{ug}_{sec}
\end{equation}
as $(C^{gg}_{sec}+\sigma_o^2)(y-G(u^{(1)}))=1$. That is, the increment of $u^{(1)}$ is $C^{ug}_{sec}$. The updated $u^{(1)}$ belongs to the linear span of the previous ensemble if and only if the increment $C^{ug}_{sec}$ is in the same subspace spanned by the previous ensemble, which is equivalent to an existence of $\alpha^{(1)}, \alpha^{(2)}$, and $\alpha^{(3)}$ such that
\begin{equation}
\begin{pmatrix}\frac{2}{9}\\\frac{-\sqrt{3}}{6}\\\frac{-1}{18}\\\frac{-1}{18}\end{pmatrix}=C^{ug}_{sec}=\sum_k\alpha^{(k)}u^{(k)}
=\begin{pmatrix}1&0&0\\-1&1&0\\0&1&0\\0&0&1\end{pmatrix}\begin{pmatrix}a^{(1)}\\a^{(2)}\\a^{(3)}\end{pmatrix},
\end{equation}
which does not have a solution. Therefore, the updated ensemble does not belong to the space spanned by the previous ensemble.

\section{Numerical experiments}\label{sec:numerical}
In this section, we apply EKI with SEC to a suite of numerical experiments to check the effectiveness of SEC for EKI with a small ensemble size. The tests include: i) a toy problem with the identity forward model, ii) compressive sensing, iii) image deblurring, iv) state estimation of the Lorenz 96 model and v) a PDE-constrained inverse problem. 
The first three problems have linear forward models, and thus other optimization methods, such as convex optimization methods, can perform well for those problems. Our focus is to compare the effect of SEC in EKI rather than advocating the use of EKI (with SEC) for the test problems considered here. We refer to \cite{lpEKI} for comparing EKI and a convex optimization method for a compressive sensing test.

To run each test, the ensemble is initialized from a Gaussian distribution with a diagonal covariance matrix (we will specify the mean and the variance later for each experiment). 
Each experiment uses several ensemble sizes that are larger than and smaller than the dimension of $u$. The sampling error correction power $a$ in \eqref{eq:correctionftn}, a tunable parameter, is chosen to provide the best result. 
When the measurement dimension, $M$, is comparable to the dimension of $u$, $N$, we apply EKI with SEC in Algorithm \ref{algo:SECEKI}. If $M<N$, the inverse problem is ill-posed, and a regularization plays a crucial role in being well-posed. In this study, we apply SEC to the $l_p$ regularized EKI \cite{lpEKI} (see the appendix for the description of $l_p$EKI with SEC used in this section). In a regularized inverse problem, the regularization coefficient ($\lambda$ in \eqref{eq:lpinv}), which determines the strength of the regularization penalty term, is crucial to obtain robust performance. It is possible to estimate the regularization coefficient, for example, using cross-validation. To minimize the effect of such estimation for the coefficient, we find the regularization by hand-tuning, which yields the best performance. In particular, we tune the regularization coefficient for the large ensemble size and use the same parameter for the small ensemble size case with SEC to check the effect of SEC directly.

In addition to covariance localization, covariance inflation has been an effective method to handle several errors in ensemble filters, such as model error and sampling error. 
The effect of the inflation in EKI has also been investigated, including adaptive inflation in the continuous-time limit \cite{analysisEKI}. To minimize the combinatorial complexity in testing the effects of various inflation methods, we focus on SEC without inflation in this section. We leave the in-depth investigation of the performance of the combined inflation and localization as future work.

\subsection{A toy problem with the identity forward model}
The first problem is an optimization problem for $u\in\mathbb{R}^{100}$
\begin{equation}\label{eq:test:toy}
\argmin_{u\in\mathbb{R}^{100}} \|u-\mathbbm{1}\|_2^2
\end{equation}
where $\mathbbm{1}=(1,1,...,1)\in\mathbb{R}^{100}$. This optimization problem is equivalent to solving \eqref{eq:inverseproblem} where the forward model $G$ is the identity map, $y=\mathbbm{1}$, and $\eta$ is Gaussian with mean zero and a diagonal covariance. The observation error variance is set to $0.1$. We set the initial ensemble mean to $(0,1,1,...,1)$ so that only the first component is different from the true solution $\mathbbm{1}$. The initial ensemble variance is 0.1 for all components. If a gradient-based method is applied to solve \eqref{eq:test:toy}, only the first component will be updated over iterations as the other components are already stationary. However, in EKI without SEC, the other components can fluctuate over iterations due to the spurious correlations related to the first component.

\begin{figure}[!ht]
\centering
\includegraphics[width=.44\textwidth]{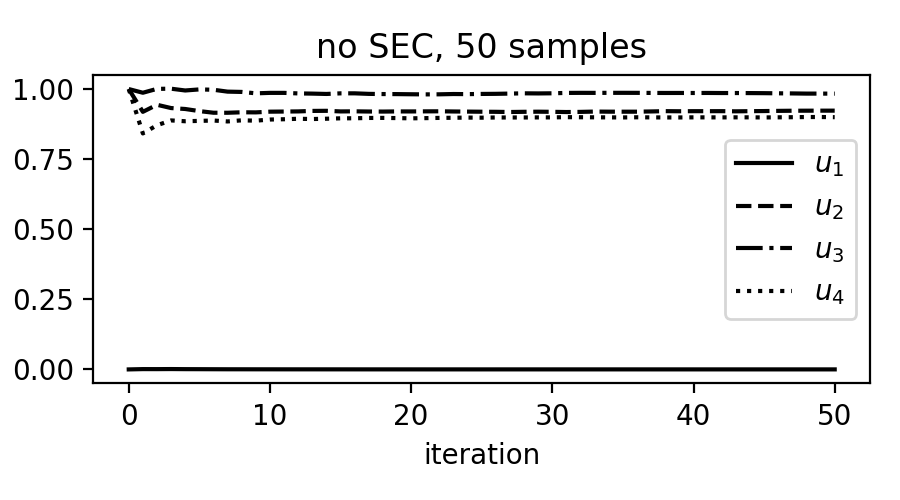}
\includegraphics[width=.44\textwidth]{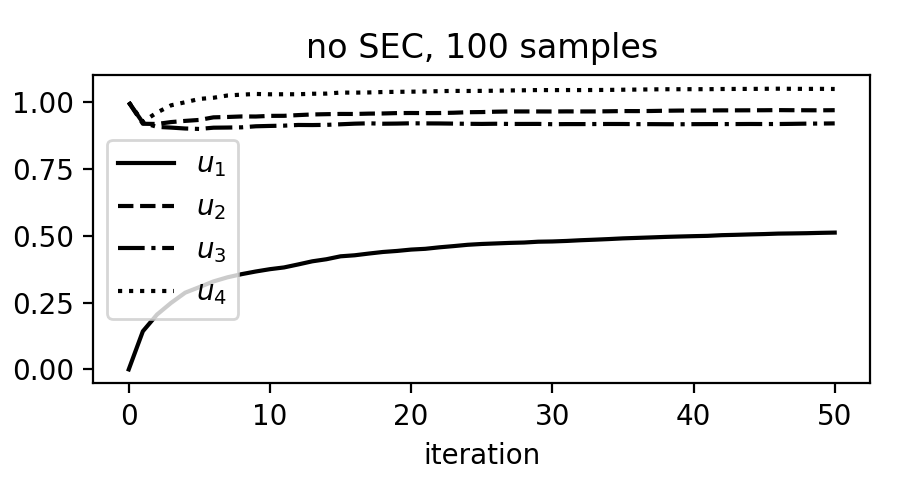}\\
\includegraphics[width=.44\textwidth]{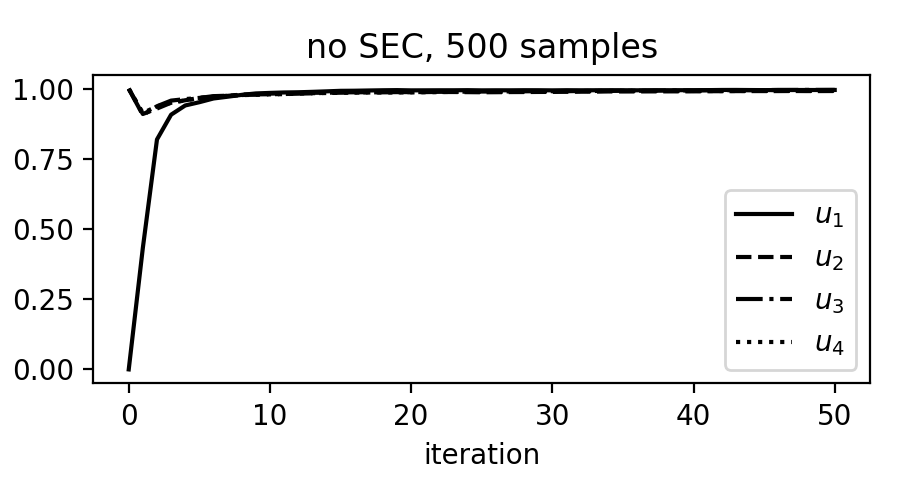}
\includegraphics[width=.44\textwidth]{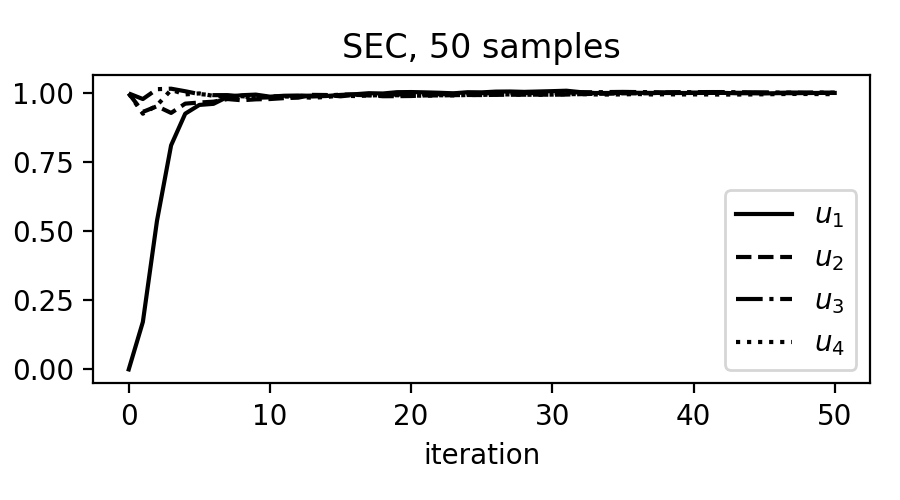}
\caption{Estimates of $u_i, i=1,2,3,4,$ for the toy problem with the identity forward model. Top left, top right, and bottom left correspond to EKI without SEC using 50, 100, and 500 samples, respectively. Bottom left is the result of EKI with SEC using 50 samples.}
\label{fig:test:toy:estimate}
\end{figure}
Figure \ref{fig:test:toy:estimate} shows the EKI estimates of the first, second, third, and fourth components of $u$, $u_1,u_2,u_3,u_4$, over iterations using several ensemble sizes 100 and 50 without SEC and an ensemble size 50 with SEC (bottom left). 
If the ensemble size is 50, which is smaller than the dimension of $u$, the $u_1$ estimate by EKI without SEC does not capture the correct value. The other components are affected by $u_1$, and they deviate from the true value although they started from the true value. When the ensemble size increases, the performance of EKI improves. The top right and bottom left plots correspond to EKI without SEC using 100 and 500 ensemble members. When the ensemble size is 100 that equals the dimension of $u$, the $u_1$ estimate moves toward 1, but the convergence rate is slow. If the ensemble size increases to 500, the $u_1$ estimate quickly converges to the true value while the other components make more minor deviations from the true value. This result shows that the sampling error due to a small ensemble size significantly degrades the performance of EKI. The bottom left plot of Figure \ref{fig:test:toy:estimate} shows the EKI result with SEC. Using only 50 ensemble members, the $u_1$ estimate by EKI with SEC converges to the true value. Including $u_1$, all components converge to the true values within ten iterations. In this test, as the forward model is the identity map that does not mix different components as a measurement, the ideal covariance matrix for a large ensemble limit would also be diagonal (as there will be no correlation between other components). If the power of the correction factor function $a$ is large, all sample correlation values with a magnitude less than one will decrease in their magnitudes, which approximates the diagonal matrix for the covariance. In our experiment, any power greater than 0.8 provides a robust result (the bottom left plot uses $a=1$).

\subsection{Compressive sensing}
The second test is a compressive sensing problem to recover an unknown vector $u\in\mathbb{R}^{100}$. The forward model is given by a $30\times 100$ matrix applied to $u$
\begin{equation}
G(u)=Au,
\end{equation}
which yields a measurement $y$ in $\mathbb{R}^{30}$. The matrix $A$ is a random Gaussian matrix drawn from the standard normal distribution. 
The current test is different from the previous test in several aspects. First, the measurement has mixed contributions from different components due to the random matrix $A$, which results in non-trivial correlations between other components. For this reason, it is not straightforward to define a distance between the components of the measurement and $u$ for covariance localization. Additionally, the measurement's dimension is smaller than the dimension of $u$, which requires an appropriate regularization to obtain a well-posed problem and a corresponding unique solution. The true signal has only four nonzero components out of 100 components. We use the $l_1$ regularization to recover the sparse structure of $u$ using $l_p$EKI with SEC (see Appendix for the algorithm of $l_p$EKI with SEC). 

A similar setup has been used in \cite{lpEKI} to test the effect of several ensemble sizes in EKI without SEC. The idea used in \cite{lpEKI} for a small ensemble size is the multiple batch run \cite{batch} that eliminates unnecessary components after each batch. Instead of the multiple batch strategy that works for only sparse recovery problems, we test the effect of SEC to recover a sparse signal without the multiple batch run.
The ensemble is initialized with mean zero and variance 1 for all components. Also, the measurement error variance is set to $10^{-2}$. The regularization coefficient, $\lambda$, is tuned based on the result of EKI without SEC, which yields $\lambda=50$. For the EKI with SEC, we use the same regularization coefficient without adjusting to check the performance difference when SEC is applied.

The top row of Figure \ref{fig:test:cs:estimate} shows the EKI result without SEC using 2000 samples. The left column is the $u$ estimate, and the right column shows the time series of the $l_1$ error and the data misfit. The ensemble size 2000 is large enough that the EKI estimate without SEC captures the three most significant components of $u$. The minor magnitude component missed by EKI is challenging to capture; the magnitude is of an order of 0.1 that is comparable to the measurement standard deviation. If the ensemble size decreases to 50, the EKI estimate without SEC degrades. EKI barely captures the most significant component while all the other components are missed (second row of Figure \ref{fig:test:cs:estimate}). Also, the estimate show fluctuations in the zero components of the true signal. The EKI result with SEC using an ensemble size 50 is shown in the bottom row of Figure \ref{fig:test:cs:estimate}. The power for the sampling error correction factor function is $a=1$. Compared with the second row, which does not use SEC, the result with SEC shows performance comparable to the large ensemble size case at the cost of a weaker magnitude in the second most significant component. Also, the convergence rate is slower than the large ensemble case, but all estimates converge within 15 iterations.

\begin{figure}[!ht]
\centering
\includegraphics[width=.95\textwidth]{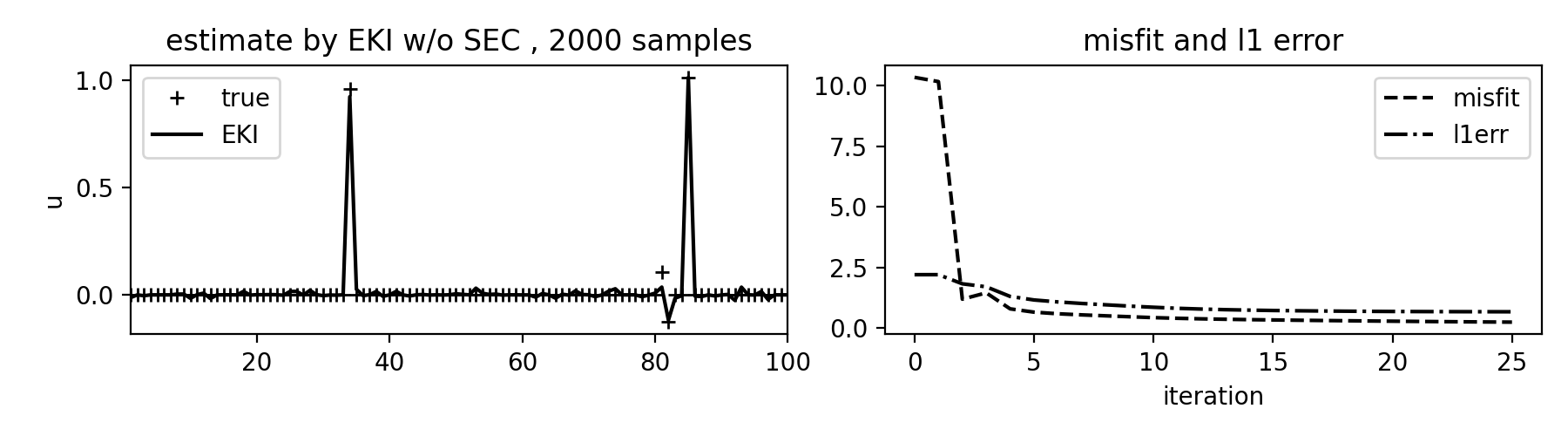}\\
\includegraphics[width=.95\textwidth]{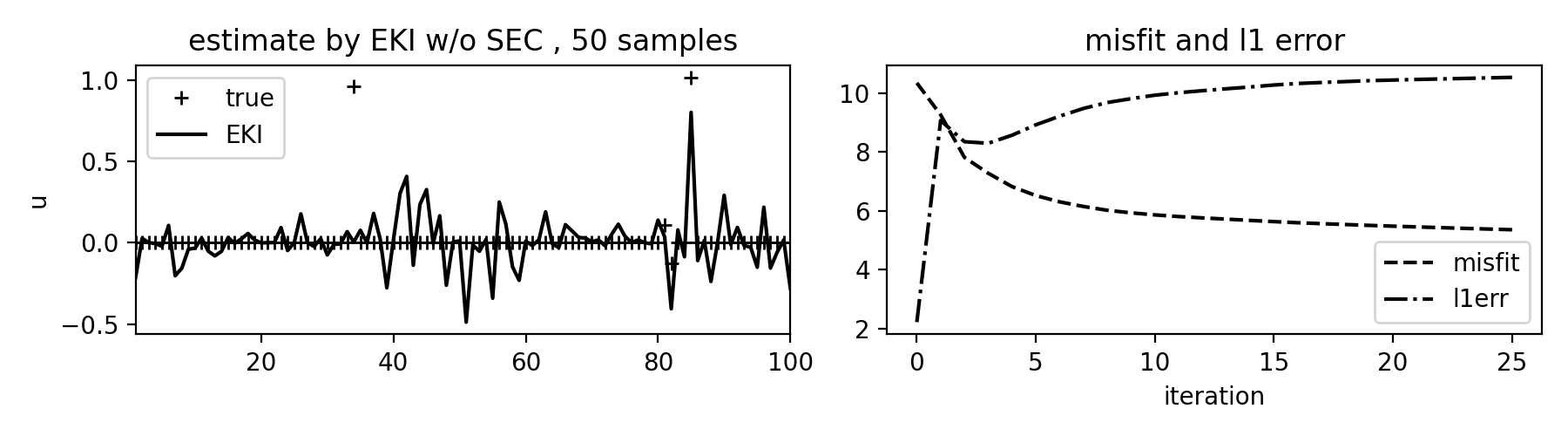}\\
\includegraphics[width=.95\textwidth]{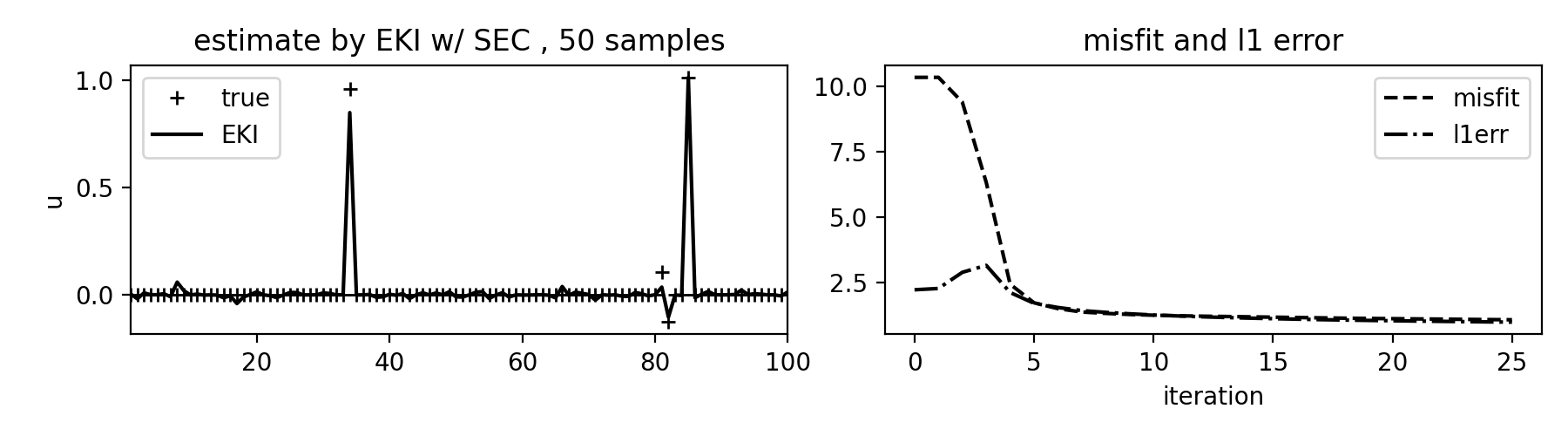}
\caption{Compressive sensing using EKI with $l_{1}$ regularization. Top: no SEC with 2000 samples. Middle: no SEC with 50 samples. Bottom: SEC with 50 samples. Left: $u$ estimates after 20 iterations. Right: time series of $l_1$ error and data misfit.}
\label{fig:test:cs:estimate}
\end{figure}

\subsection{Image deblurring}
We consider one more inverse problem with a linear forward model, image deblurring. The forward model is a Gaussian filter with a standard deviation set to 0.7. The measurement error variance is relatively small, $10^{-4}$, to focus on the deblurring performance of EKI. The true signal is the cameraman image of size $128\times 128$, and the measurement is also an image of the same size. Thus, we have $N=128^2$ and $M=128^2$. We use the discrete-time EKI without $l_p$ regularization, and the ensemble is initialized from a Gaussian distribution with mean zero and variance $2\times 10^{-4}$.

\begin{figure}[!ht]
\centering
\includegraphics[width=1\textwidth]{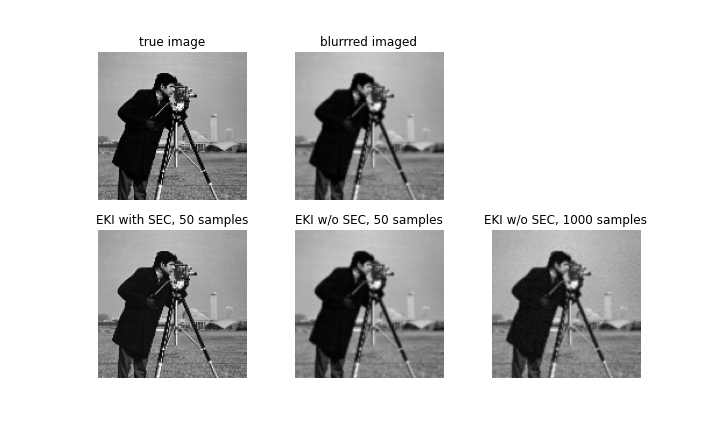}
\caption{Image deblurring of an 128x128 image. Top left: true image. Top center: measurement. Bottom left: EKI with SEC, 50 samples. Bottom center: EKI without SEC, 50 samples. Bottom right: EKI without SEC, 1000 samples.}\label{fig:test:deblurring}
\end{figure}
The top row of Figure \ref{fig:test:deblurring} shows the true and the blurred images. The bottom row includes the results of EKI with and without SEC for several ensemble sizes. The bottom left plot is EKI with SEC using 50 samples, which shows the best result. The sampling error correction power $a$ is set to $3$ to obtain the result. The other plots on the bottom are EKI without SEC using 50 (middle) and 1000 (right) samples after 25 iterations. Without SEC, the large sample case 1000 is not performing well; the ensemble size 1000 is still smaller than the dimension of the image $128^2$. On the other hand, EKI with SEC shows output with a significant deblurring effect using only 50 ensemble members.

\begin{figure}[!ht]
\centering
\includegraphics[width=1\textwidth]{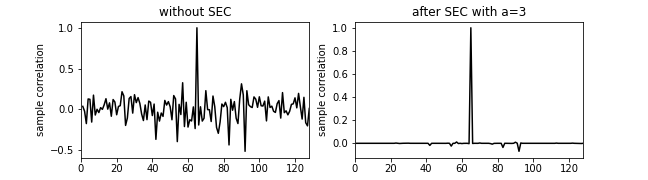}
\caption{Sample correlation between the pixel value at $(64,64)$ and $(i,64), i=1,2,...,128$. Left: without SEC. Right: after SEC with $a=3$.}
\label{fig:test:deblurring:sampcor}
\end{figure}
The sample correlations between the pixel fixed at $(64,64)$ and other pixels at $(i,64),i=1,2,...,128)$ are shown in Figure \ref{fig:test:deblurring:sampcor} alongside its corrected correlation using $a=3$ for the correction factor function. Without SEC (left plot), many of the pixels have non-trivial correlations with the pixel at $(64,64)$, and thus the measurement at $(64,64)$ affects many other pixels. On the other hand, when SEC applies to the sample correlation with $a=3$, all sample correlations except the same pixel attenuate. The sampling error correction using $a=3$ is equivalent to the standard localization with a very short cutoff distance, which is 1 in this test.

\subsection{Lorenz 96 with Fourier measurements}
We turn our interest to nonlinear forward models. The first nonlinear forward model problem is to estimate the initial value of a nonlinear dynamical system from a measurement of the state variable at a future time. The dynamical system we are interested in is the Lorenz 96 model, a standard model for a chaotic system with applications in geophysical turbulent flows \cite{Lorenz96}. In particular, we consider the 40-dimensional Lorenz 96 model with $F=8$ with the periodic boundary condition
\begin{equation}\label{eq:test:l96}
\frac{dx_i}{dt}=(x_{i+1}-x_{i-2})x_{i-1}-x_i+F, \quad x(0)=x_0.
\end{equation}
The measurement values are the Fourier components ($\sin$ and $\cos$ coefficients) of the state variable $x$ at $t=0.5$ except the largest and the second-largest wave components. Therefore, the forward model is the nonlinear map from the initial value $x_0$ to the first eighteen wave components of $x$ at $t=0.5$. Note that there is no well-defined distance between the Fourier components and the initial value. As the Fourier measurement is incomplete (missing the two largest wave components), we cannot use the inverse Fourier transform to define a distance between the measurement and the initial value. That is, the standard localization is not applicable in this test.

We use a fourth-order Runge-Kutta method with a time step 0.01 to compute the state variable $x$ at $t=0.5$. 
The measurement error variance is set to a relatively small value $10^{-2}$ to focus on the nonlinear behavior of the forward model. The ensemble is initialized around 0 with variance 1. As the measurement dimension, $M=36$, is smaller than the dimension of the initial value, $N=40$, we use the $l_p$EKI with $p=2$, which is equivalent to Tikhonov EKI \cite{TEKI}. The hand-tuned regularization coefficient $\lambda$ is 0.1.

To check the effect of SEC in EKI, we compare the estimates of $x_0$ by EKI without SEC using 1000 and 30 samples and EKI with SEC using 30 samples shown in Figure \ref{fig:test:l96:estimate} along with the time series of the $l_1$ error and data misfit. The large sample case without SEC has the best result. The $l_1$ error and the data misfit monotonically decrease with limits 45.51 and 8.65, respectively. The superior performance of the large sample case is expected; 1000 sample is 25 times larger than the dimension of $x_0$. The second-best result is the EKI with SEC using 30 samples and the sampling error correction power $a=1$. The $l_1$ error increases at the beginning period but quickly decreases down to 47.15, which is marginally worse than the large size case. The EKI without SEC using 30 samples, on the other hand, does not show a robust estimation result. The $l_1$ error increases and converges to 123, which is worse than the initial guess error. The converged data misfit is smaller than the initial misfit but is still worse than the other two methods. 
\begin{figure}[!htbp]
\centering
\includegraphics[width=.8\textwidth]{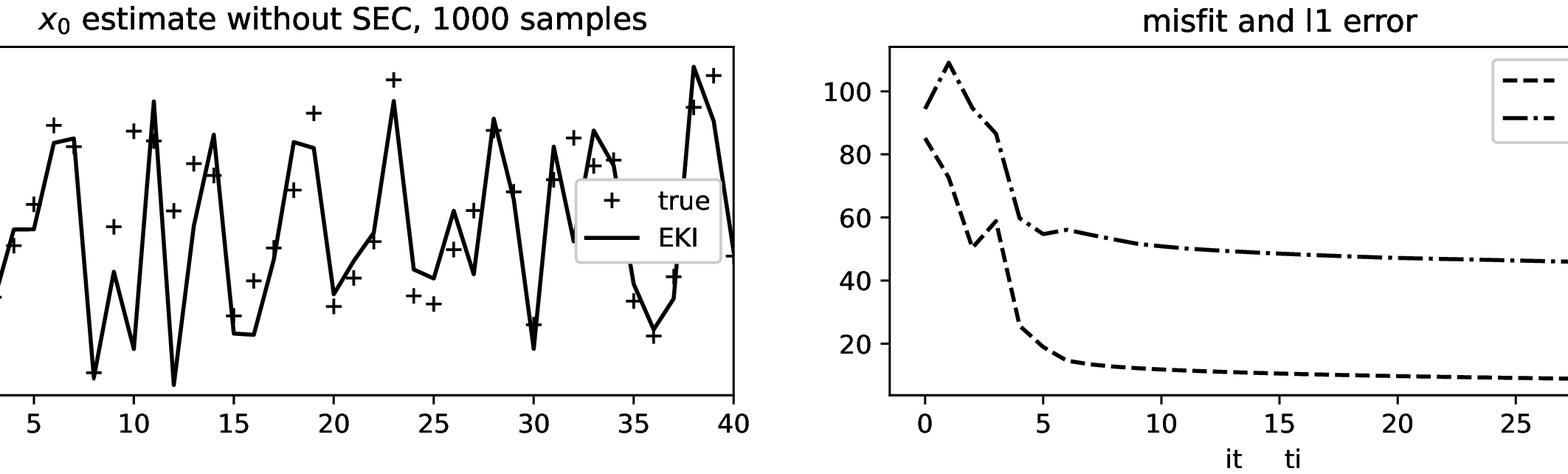}\\
\includegraphics[width=.8\textwidth]{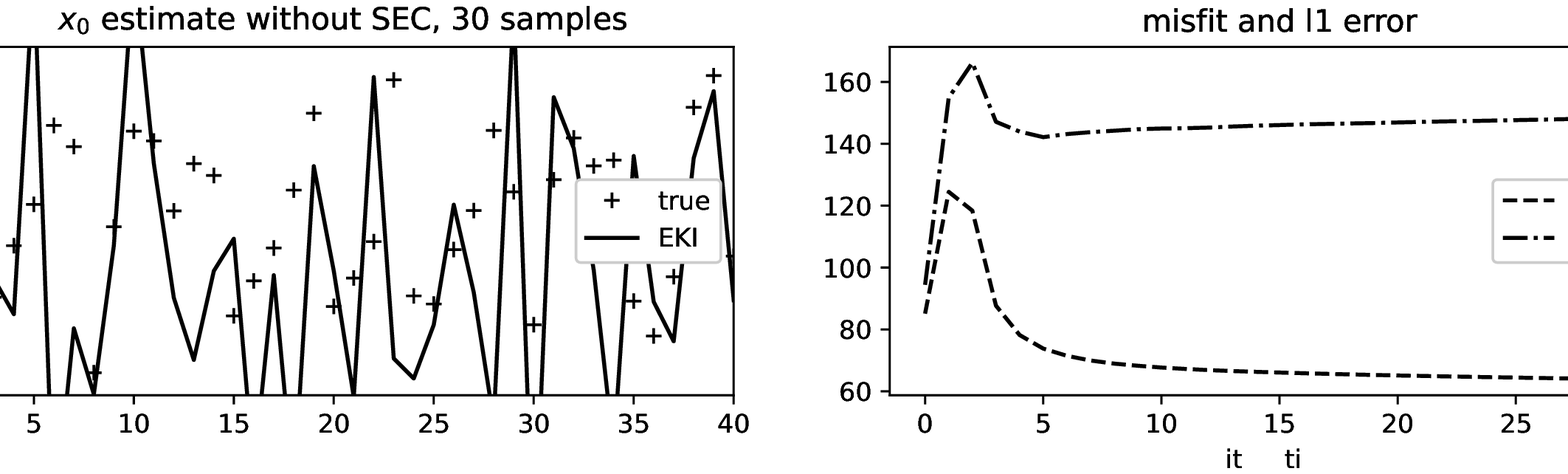}
\includegraphics[width=.8\textwidth]{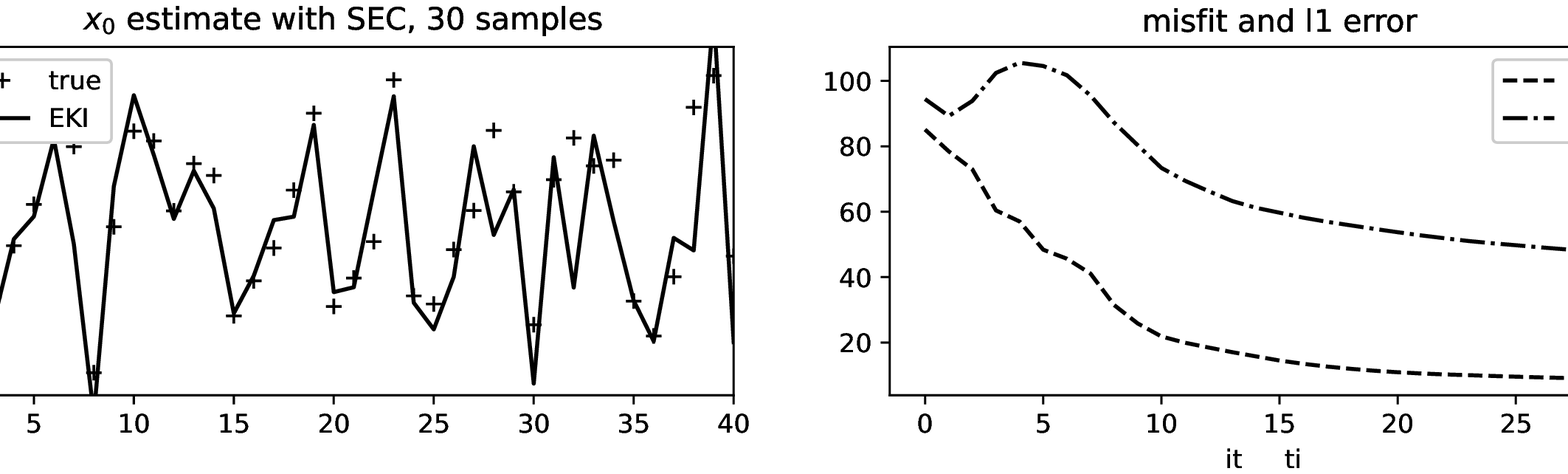}\\
\caption{Estimation of the initial value of the Lorenz 96 from partial Fourier measurement of a future state. Top and middle: EKI without SEC using 1000 and 30 samples. Bottom: EKI with SEC using 50 samples.}
\label{fig:test:l96:estimate}
\end{figure}





\subsection{2D elliptic inverse problem with discontinuous coefficient}
The last test is an inverse problem in which the forward model is related to solving an elliptic PDE. The following elliptic model describes the subsurface flow described by the Darcy's law in the two-dimensional unit square $\Omega=(0,1)^2\subset\mathbb{R}^2$
\begin{equation}\label{eq:test:2d}
-\nabla \cdot (k(x)\nabla p(x)) = f(x), \quad x=(x_1,x_2)\in(0,1)^2.
\end{equation}
Here $k(x)\geq\alpha>0$ is a scalar permeability field, and $p(x)$ is the pressure field (or the piezometric head) of the flow. The right hand side $f(x)$ is related to an external source. Our interest is to estimate the permeability field $k(x)$ from partial and noisy measurements of the pressure field, which is a standard inverse problem setup in oil reservoir simulations. 

We follow a similar setup in \cite{lpEKI} for the source term and the boundary condition (see \cite{2dsetup} for a physical motivation of this setup). First, the source term is piecewise constant in the $y$-direction
\begin{equation}
f(x_1,x_2)=\left\{
\begin{array}{ll}
0&\mbox{if } 0\leq x_2\leq \frac{4}{6}\\
137&\mbox{if }\frac{4}{6}< x_2\leq \frac{5}{6},\\
274&\mbox{if }\frac{5}{6}<x_2\leq 1.
\end{array}\right.
\end{equation}
The boundary condition is mixed with Dirichlet and Neumann conditions
\begin{equation}p(x_1,0)=100,  \frac{\partial p}{\partial x_1}(1,x_2)=0, -k\frac{\partial p}{\partial x_1}(0,x_2)=500,  \frac{\partial p}{\partial x_2}(x_1,1)=0,
\end{equation}
The measurement is the pressure value at $20\times 20$ uniformly located grid points with a measurement error variance $10^{-6}$ that is uncorrelated between different locations. When a permeability field $k(x)$ is given, we use the FEM with the second-order polynomial basis on a $50\times 50$ uniform mesh. To guarantee that the permeability is positive, we use the log permeability $\log p(x)$ as the unknown variable $u$ of the inverse problem. 

The true log-permeability is a square-shaped discontinuous function with only two values, 0 and 1 (top left plot of Figure \ref{fig:test:2dpermeability}). As an initial guess, we use the value obtained by applying a Gaussian filter with a standard deviation of 5 to the true value (top right plot of Figure \ref{fig:test:2dpermeability}). The ensemble is initialized with the blurred field as the mean, and the ensemble variance is set to $10^{-3}$. The measurement's dimension, $20\times 20=400$, is much smaller than the dimension of the unknown variable, $50\times 50=2500$. As the solution shows a sparse structure in the gradient space, a TV penalty term will be appropriate to regularize the inverse problem. However, there is no TV regularization implementation in EKI yet, and thus we use the $l_1$ regularization for $u$ using $l_p$EKI. 

The second, third, and fourth rows of Figure \ref{fig:test:2dpermeability} show the EKI estimates of $u$ using various ensemble sizes with and without SEC (the SEC results are in the right column). Using 2000 samples (second row), SEC does not make a significant difference. Both estimates show comparable results capturing edges on the top and left of the permeability field while the right and bottom parts show more noisy edges. The performance degrades as the sample size decreases. Without SEC, EKI loses all sharp edges with noisy boundaries. 
As the ensemble size decreases, EKI with SEC shows superior performance than the EKI without SEC. Except for the bottom right corner, the EKI with SEC shows edges of the square region more clear than EKI without SEC. The effect of the sampling error correction power $a$ is sensitive in this test compared to the other tests. The optimal power $a$ is $0.2$ obtained by trying values from 0 to 1 spaced by 0.1 (the power was tuned only for the 300 sample case and used the same value for the other sample sizes).

\begin{figure}[!htbp]
\centering
\includegraphics[width=.45\textwidth]{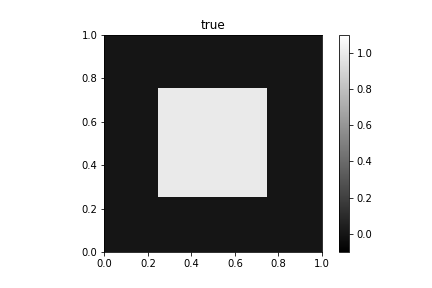}
\includegraphics[width=.45\textwidth]{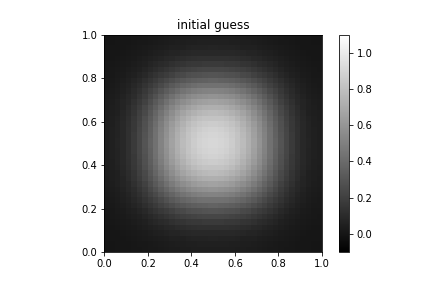}\\
\includegraphics[width=.45\textwidth]{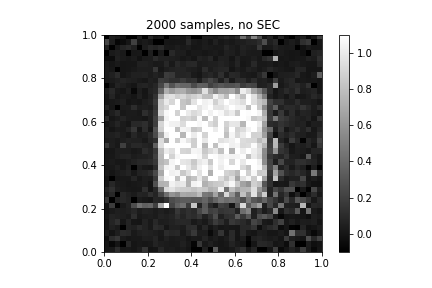}
\includegraphics[width=.45\textwidth]{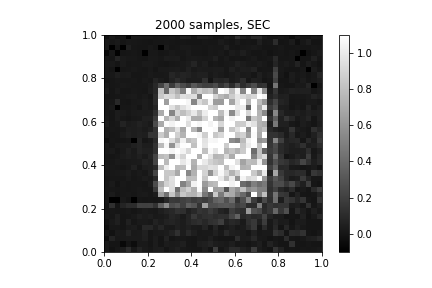}\\
\includegraphics[width=.45\textwidth]{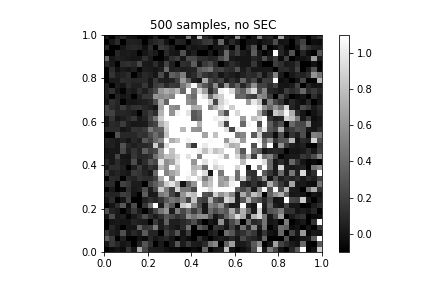}
\includegraphics[width=.45\textwidth]{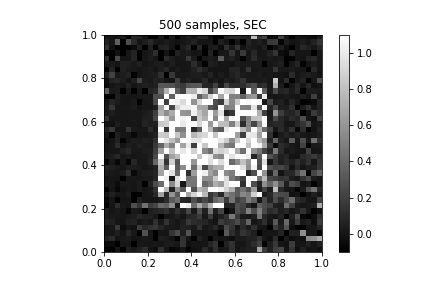}\\
\includegraphics[width=.45\textwidth]{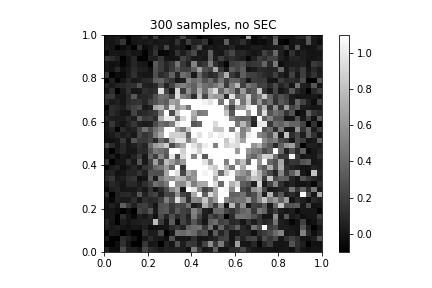}
\includegraphics[width=.45\textwidth]{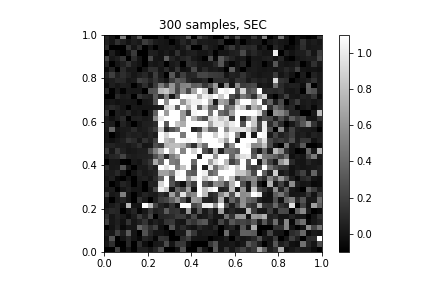}
\caption{Log-permeability field. Top row: true and the initial guess. 2nd row: EKI with samples. 3rd row: EKI with 500 samples. 4th row: EKI with 300 samples. Starting from the 2nd row, left column: without SEC. Right column: with SEC.}
\label{fig:test:2dpermeability}
\end{figure}

\section{Discussions and conclusions}\label{sec:conclusions}
This study proposed a sampling error correction (SEC) method to enable a small ensemble size in ensemble Kalman inversion (EKI). We validated the effectiveness and robustness of the proposed method through a series of inverse problems, including several standard linear forward model problems and nonlinear forward model problems. The proposed SEC method does not require a well-defined distance to determine the impact of measurement on updating an unknown variable used in the standard localization. Instead, the SEC method uses only the sample correlation value to correct the statistical deficiency induced by small sample size. The SEC method for EKI uses a power function \eqref{eq:correctionftn} to attenuate the magnitude of the sample correlation that is not perfectly correlated. Similar to the localization in ensemble Kalman filters, we showed that the SEC strategy modifies the ensemble so that the ensemble does not necessarily satisfy the invariant subspace property of the ensemble update. Due to this violation of the invariance property, the choice of an initial ensemble is less critical than the EKI without SEC, where the solution must belong to the linear span of the initial ensemble.

The proposed method needs to specify a parameter, $a$ of \eqref{eq:correctionftn}, to obtain robust estimation results. In the current study, we only tested hand-tuned values. This parameter can be estimated in advance using other techniques, such as cross-validation. We leave the investigation of this direction as future work. It would also be natural to investigate different types of SEC. For example, the localization method based on observation sensitivity matrices used in a petroleum reservoir application \cite{EmerackReynolds} is worth further investigation in the context of SEC for EKI.

We also considered only a discrete-time EKI. Other variants of EKI, such as the continuous-time limit and its adaptive time-stepping, have shown robust results in many stringent test problems along with their computational efficiency in some instances \cite{analysisEKI,TEKI}. Moreover, we used the ensemble Kalman filter using the perturbed measurement, leaving other ensemble-based Kalman filters, such as ensemble square-root filters, as future work. We plan to study SEC's application in different variants of EKI and compare the performance, which will be reported in another place. The covariance inflation and localization are the two most popular methods to handle statistical deficiencies for the ensemble-based Kalman filters. In most of the numerical experiments tests here, we were able to see the performance increase using SEC without inflation. In our numerical experiments, we found certain problems requiring both inflation and SEC to obtain robust performance for EKI (for example, the Lorenz 96 test requires inflation along with SEC when the measurement size is small). It would be interesting to investigate the interplay of inflation and SEC to improve the performance of EKI.

\subsection*{Acknowledgements}
The author is supported by NSF DMS-1912999 and ONR MURI N00014-20-1-2595.

\setcounter{section}{1}
\section*{Appendix: $l_p$-regularized EKI}
The idea of the $l_p$-regularized EKI ($l_p$EKI) is to solve the Tikhonov EKI \cite{TEKI} through a transformation of the variable $u$. 
For a fixed $p>0$, we define a transformation $\Psi:\mathbb{R}^N\to\mathbb{R}^N$
\begin{equation}\label{eq:vtou}
\Psi(u)=(\psi(u_1),\psi(u_2),...,\psi(u_N))
\end{equation}
where $\psi:\mathbb{R}\to\mathbb{R}$ is defined as
\begin{equation}
\psi(x)=\mbox{sgn}(x)|x|^{\frac{2}{p}}.
\end{equation}
Using the transformed variable $v=\Psi(u)$, $l_p$EKI uses the Tikhonov EKI \cite{TEKI} to solve the following $l_2$-regularized optimization problem
\begin{equation}\label{eq:l2inv}
\mbox{argmin}_{v\in \mathbb{R}^N}\;\lambda\|v\|_2^2+\|y-\tilde{G}(v)\|^2_{\Gamma},
\end{equation}
where $\tilde{G}$ is the pullback of $G$ by $\Xi:=\Psi^{-1}$ (note that $\Psi$ is invertible)
\begin{equation}
\tilde{G}=G\circ \Xi, \quad \Xi=\Psi^{-1}.
\end{equation}
From the definition of $\Psi$, it is straightforward to check that \eqref{eq:l2inv} is equivalent to 
\begin{equation}\label{eq:lpinv}
\mbox{argmin}_{u\in \mathbb{R}^N}\;\lambda\|u\|_p^p+\|y-G(u)\|^2_{\Gamma}.
\end{equation}
This equivalence holds for any $p>0$ but there is lower bound on $p$ due to numerical stability. If $p$ is too small, $\xi_p$ can blow up due to a large exponent $\frac{2}{p}$. 

The main idea to solve the Tikhonov regularization is to incorporate an augmented measurement vector $z\in\mathbb{R}^{N+M}$
\begin{equation}
z=(y,0)
\end{equation} 
along with an augmented forward model $F(v):\mathbb{R}^{N}\to\mathbb{R}^{N+M}$
\begin{equation}
F(v)=(\tilde{G}(v),v).
\end{equation}
Note that there is no computational cost increase in the augmented forward model. Using this augmented system, we have an inverse problem of estimating $v\in\mathbb{R}^N$ from $z\in\mathbb{R}^{N+M}$
\begin{equation}\label{eq:augobs}
z=F(v)+\zeta
\end{equation}
where $\zeta$ is a $N+M$-dimensional measurement error for the augmented system, which is Gaussian with mean zero and covariance
\begin{equation}\label{eq:augmentedcovariance}
\Sigma=\begin{pmatrix}\Gamma&0\\0&\frac{1}{\lambda}I_{N}\end{pmatrix}
\end{equation}
for the $N\times N$ identity matrix $I_N$.

The estimation of $u$ from $z$ can be done by solving the following optimization problem using the standard EKI
\begin{equation}\label{eq:auginv}
\mbox{argmin}_{u}\|z-F(v)\|_{\Sigma}^2,
\end{equation}
which is equivalent to the Tikhonov regularization problem \eqref{eq:l2inv}. Once \eqref{eq:auginv} is solved with a solution, say $v^{\dag}$, the minimizer of the $l_p$ regularization problem \eqref{eq:lpinv} is estimated by $u^{\dag}=\Xi(v^{\dag})$.

The complete algorithm of $l_p$EKI is given below.

\vspace{0.05\textwidth}
\begin{algorithm}\label{algo:lpEKI}
\textbf{$l_p$-regularized EKI with SEC}
\end{algorithm}
Assumption: an initial ensemble of size $K$, $\{v_0^{(k)}\}_{k=1}^K$, is given.\\
For $n=1,2,...,$
\begin{enumerate}
	\item Prediction step using the forward model:
	\begin{enumerate}
		\item Apply the augmented forward model $F$ to each ensemble member
\begin{equation}
f_n^{(k)}:=F(v_n^{(k)})=(\tilde{G}(v_n^{(k)}), v_n^{(k)})
\end{equation}
		\item From the set of the predictions $\{f_n^{(k)}\}_{k=1}^K$, calculate the mean and covariances
\begin{equation}\label{eq:rEKI:samplemean}
\overline{f}_n=\frac{1}{K}\sum_{k=1}^Kf_n^{(k)},
\end{equation}
\begin{equation}\label{eq:rEKI:samplecovariance}
\begin{split}
C^{vf}_n&=\frac{1}{K}\sum_{k=1}^K(v_n^{(k)}-\overline{v}_n)\otimes(f_n^{(k)}-\overline{f}_n),\\
C^{ff}_n&=\frac{1}{K}\sum_{k=1}^K(f_n^{(k)}-\overline{f}_n)\otimes(f_n^{(k)}-\overline{f}_n).
\end{split}
\end{equation}
Here $\overline{v}_n$ is the ensemble mean of $\{v_n^{(k)}\}$, i.e., 
\begin{equation}
\overline{v}_n=\displaystyle\frac{1}{K}\sum_{k=1}^Kv_n^{(k)}.
\end{equation}

\item Calculate the sampling error-corrected covariance matrices
	\begin{equation}
	\begin{split}
	C^{vf}_{n,sec}&=v^uR^{vf}_{sec}v^g,\\
C^{ff}_{n,sec}&=v^gR^{ff}_{sec}v^g.
	\end{split}
	\end{equation}
	where $R^{vf}_{sec}$ and $R^{ff}_{sec}$ are the matrices obtained by increasing the
	power of each element of the correlation matrices of $C^{vf}_{n}$ and $C^{ff}_{n}$ to $a+1$.
	\end{enumerate}

\item Analysis step:
	\begin{enumerate}
		\item Update each ensemble member $v_n^{(k)}$ using perturbed measurements $z_{n+1}^{(k)}=z+\zeta_{n+1}^{(k)}$ where $\zeta_{n+1}^{(k)}$ is Gaussian with mean zero and covariance $\Sigma$
\begin{equation}\label{eq:rEKI:ensembleupdate}
v_{n+1}^{(k)}=v_{n}^{(k)}+C^{vf}_{n,sec}(C^{ff}_{n,sec}+\Sigma)^{-1}(z_{n+1}^{(k)}-f_n^{(k)}).
\end{equation}

		\item The $n$-th step $l_p$EKI estimate of $u$, $u_n$, is given by
	\begin{equation}\label{eq:rEKI:finalestimate}
	u_n = \Xi(\overline{v}_n).
	\end{equation}
	\end{enumerate}
	\item Repeat steps 1 and 2 until $u_n$ converges.
\end{enumerate}

The additional computational cost of $l_p$EKI in comparison with the standard EKI without using the augmented system is the inversion of the $(N+M)\times(N+M)$ matrix $C_n^{ff}+\Gamma$ (the standard EKI without the augmented system uses a $m\times m$ covariance matrix). As the matrix is symmetric positive definite, the inversion can be achieved efficiently.

\bibliographystyle{plain}
\bibliography{secEKI}

\begin{thebibliography}{10}

\bibitem{EAKF}
Jeffrey~L. Anderson.
\newblock An ensemble adjustment kalman filter for data assimilation.
\newblock {\em Monthly Weather Review}, 129(12):2884--2903, 2001.

\bibitem{AndersonSEC}
Jeffrey~L. Anderson.
\newblock Localization and sampling error correction in ensemble kalman filter
  data assimilation.
\newblock {\em Monthly Weather Review}, 140(7):2359 -- 2371, 2012.

\bibitem{inflation}
Jeffrey~L. Anderson and Stephen~L. Anderson.
\newblock A monte carlo implementation of the nonlinear filtering problem to
  produce ensemble assimilations and forecasts.
\newblock {\em Monthly Weather Review}, 127(12):2741 -- 2758, 1999.

\bibitem{locEKF}
Kay Bergemann and Sebastian Reich.
\newblock A localization technique for ensemble kalman filters.
\newblock {\em Quarterly Journal of the Royal Meteorological Society},
  136(648):701--707, 2010.

\bibitem{ETKF}
Craig~H. Bishop, Brian~J. Etherton, and Sharanya~J. Majumdar.
\newblock Adaptive sampling with the ensemble transform kalman filter. part i:
  Theoretical aspects.
\newblock {\em Monthly Weather Review}, 129(3):420 -- 436, 2001.

\bibitem{2dsetup}
Jesus Carrera and Shlomo~P. Neuman.
\newblock Estimation of aquifer parameters under transient and steady state
  conditions: 3. application to synthetic and field data.
\newblock {\em Water Resources Research}, 22(2):228--242, 1986.

\bibitem{TEKI}
Neil~K. Chada, Andrew~M. Stuart, and Xin~T. Tong.
\newblock Tikhonov regularization within ensemble kalman inversion.
\newblock {\em SIAM Journal on Numerical Analysis}, 58(2):1263--1294, 2020.

\bibitem{EmerackReynolds}
Alexandre Emerick and Albert Reynolds.
\newblock Combining sensitivities and prior information for covariance
  localization in the ensemble kalman filter for petroleum reservoir
  applications.
\newblock {\em Comput Geosci}, 15:251, 2011.

\bibitem{EnKF}
G~Evensen.
\newblock {\em Data Assimilation: The Ensemble Kalman Filter}.
\newblock Springer, London, 2009.

\bibitem{Gaspari-Cohn}
Gregory Gaspari and Stephen~E. Cohn.
\newblock Construction of correlation functions in two and three dimensions.
\newblock {\em Quarterly Journal of the Royal Meteorological Society},
  125(554):723--757, 1999.

\bibitem{locHamill}
Thomas~M. Hamill, Jeffrey~S. Whitaker, and Chris Snyder.
\newblock Distance-dependent filtering of background error covariance estimates
  in an ensemble kalman filter.
\newblock {\em Monthly Weather Review}, 129(11):2776 -- 2790, 2001.

\bibitem{locHoutekamer}
P.~L. Houtekamer and Herschel~L. Mitchell.
\newblock A sequential ensemble kalman filter for atmospheric data
  assimilation.
\newblock {\em Monthly Weather Review}, 129(1):123 -- 137, 2001.

\bibitem{EKI}
M.A. Iglesias, Kody J.~H. Law, and A.M. Stuart.
\newblock Ensemble kalman methods for inverse problems.
\newblock {\em Inverse Problems}, 29(4):045001, 2013.

\bibitem{iterativeregularization}
Marco~A Iglesias.
\newblock A regularizing iterative ensemble kalman method for {PDE}-constrained
  inverse problems.
\newblock {\em Inverse Problems}, 32(2):025002, 2016.

\bibitem{MLEKI}
Nikola~B Kovachki and Andrew~M Stuart.
\newblock Ensemble kalman inversion: a derivative-free technique for machine
  learning tasks.
\newblock {\em Inverse Problems}, 35(9):095005, 2019.

\bibitem{lpEKI}
Yoonsang Lee.
\newblock $l_p$ regularization for ensemble kalman inversion.
\newblock 2020.
\newblock arXiv:2009.03470.

\bibitem{Lorenz96}
E.~N. Lorenz.
\newblock Predictability: a problem partly solved.
\newblock {\em Proc. Seminar on Predictability}, 1, 1996.

\bibitem{locMCMC}
M.~Morzfeld, X.T. Tong, and Y.M. Marzouk.
\newblock Localization for mcmc: sampling high-dimensional posterior
  distributions with local structure.
\newblock {\em Journal of Computational Physics}, 380:1--28, 2019.

\bibitem{multipleBatch}
Hayden Schaeffer.
\newblock Learning partial differential equations via data discovery and sparse
  optimization.
\newblock {\em Proceedings of the Royal Society A: Mathematical, Physical and
  Engineering Sciences}, 473(2197):20160446, 2017.

\bibitem{batch}
Hayden Schaeffer.
\newblock Learning partial differential equations via data discovery and sparse
  optimization.
\newblock {\em Proceedings of the Royal Society A: Mathematical, Physical and
  Engineering Sciences}, 473(2197):20160446, 2017.

\bibitem{analysisEKI}
Claudia Schillings and Andrew~M. Stuart.
\newblock Analysis of the ensemble kalman filter for inverse problems.
\newblock {\em SIAM Journal on Numerical Analysis}, 55(3):1264--1290, 2017.

\bibitem{sparseEKI}
Jin-Long~Wu Tapio~Schneider, Andrew M.~Stuart.
\newblock Ensemble kalman inversion for sparse learning of dynamical systems
  from time-averaged data.
\newblock 2020.
\newblock arXiv:2007.06175.

\bibitem{EnSQKF}
Michael~K. Tippett, Jeffrey~L. Anderson, Craig~H. Bishop, Thomas~M. Hamill, and
  Jeffrey~S. Whitaker.
\newblock Ensemble square root filters.
\newblock {\em Monthly Weather Review}, 131(7):1485 -- 1490, 2003.

\end{thebibliography}

\end{document}